\documentclass[preprint,12pt]{elsarticle}

\usepackage{amssymb,MnSymbol}
\usepackage{amsthm,amsmath}

\journal{Eur. J. Oper. Res.}

\begin{document}

\theoremstyle{plain}
\newtheorem{theorem}{Theorem}
\newtheorem{lemma}[theorem]{Lemma}
\newtheorem{proposition}[theorem]{Proposition}
\newtheorem{corollary}[theorem]{Corollary}
\newtheorem{fact}[theorem]{Fact}
\newtheorem*{main}{Main Theorem}

\theoremstyle{definition}
\newtheorem{definition}[theorem]{Definition}
\newtheorem{example}[theorem]{Example}

\theoremstyle{remark}
\newtheorem*{conjecture}{Conjecture}
\newtheorem{remark}{Remark}
\newtheorem{claim}{Claim}

\newcommand{\N}{\mathbb{N}}
\newcommand{\R}{\mathbb{R}}
\newcommand{\I}{\mathbb{I}}
\newcommand{\Vspace}{\vspace{2ex}}
\newcommand{\bfx}{\mathbf{x}}
\newcommand{\bfy}{\mathbf{y}}
\newcommand{\bfh}{\mathbf{h}}
\newcommand{\bfe}{\mathbf{e}}
\newcommand{\Q}{Q}

\begin{frontmatter}



\title{Weighted Banzhaf power and interaction indexes through weighted approximations of games}


\author[1]{Jean-Luc Marichal}
\ead{jean-luc.marichal[at]uni.lu}

\author[1]{Pierre Mathonet}
\ead{pierre.mathonet[at]uni.lu}

\address[1]{Mathematics Research Unit, FSTC, University of Luxembourg\\ 6, rue Coudenhove-Kalergi, L-1359 Luxembourg, Luxembourg.}

\begin{abstract}
The Banzhaf power index was introduced in cooperative game theory to measure the real power of players in a game. The Banzhaf interaction index
was then proposed to measure the interaction degree inside coalitions of players. It was shown that the power and interaction indexes
can be obtained as solutions of a standard least squares approximation problem for pseudo-Boolean functions. Considering certain weighted
versions of this approximation problem, we define a class of weighted interaction indexes that generalize the Banzhaf interaction index. We
show that these indexes define a subclass of the family of probabilistic interaction indexes and study their most important properties. Finally,
we give an interpretation of the Banzhaf and Shapley interaction indexes as centers of mass of this subclass of
interaction indexes.
\end{abstract}


\end{frontmatter}



\section{Introduction}

In cooperative game theory, various kinds of power indexes are used to measure the influence that a given player has on the outcome of the
game or to define a way of sharing the benefits of the game among the players. The best known power indexes are due to
Shapley~\cite{Sha53,ShaShu54} and Banzhaf~\cite{Ban65,DubSha79}. However, there are many other examples of such indexes in the literature; see
for instance \cite{AloCasHolLor08,DeePac78,Web88}.

When one is concerned by the analysis of the behavior of players in a game, the information provided by power indexes might be far insufficient,
for instance due to the lack of information on how the players interact within the game. The notion of interaction index was then introduced to
measure an interaction degree among players in coalitions. The first proposal goes back to Owen~\cite{Owe72} who defined the ``co-value'' of a pair of players $\{i,j\}$ in a game $v$ on $N=\{1,\ldots,n\}$ as an average over all coalitions $S\subseteq N\setminus\{i,j\}$ of the quantity
$$
v(S\cup\{i,j\})-v(S\cup\{i\})-v(S\cup\{j\})+v(S).\footnote{Intuitively, this quantity measures how the players $i$ and $j$ interact in the presence of $S$. In contrast, the ``coalitional power'' of the pair $\{i,j\}$ in the presence of $S$ is measured by an average value of $v(S\cup\{i,j\})-v(S)$ (see~\cite{MarKojFuj07}).}
$$
This definition was rediscovered and interpreted as an interaction index by Murofushi and Soneda~\cite{MurSon93}. A systematic approach was then initiated by Grabisch~\cite{Gra97b,Gra97a} and Roubens~\cite{Rou96} and led to the
definition of the Shapley and Banzhaf interaction indexes as well as many others. For general background, see Fujimoto et
al.~\cite{FujKojMar06}.

There is no universal power or interaction index that can be used in every single practical situation. The choice of such an index often depends
on the problem under consideration. Several axiomatizations of power and interaction indexes have then been proposed thus far (see
\cite{DubSha79,Sha53} for power indexes and \cite{FujKojMar06,GraRou99} for interaction indexes).

In addition to being axiomatized, the Banzhaf and Shapley power indexes were shown to be solutions of simple least squares approximation
problems:

\begin{itemize}
\item Charnes et al.~\cite{ChaGolKeaRou88} considered the problem of finding the best efficient (hence constrained) approximation of a given
game by an additive game in the sense of weighted least squares. They showed that the Shapley power index appears as the unique solution of the
approximation problem for a specified choice of the weight system over the coalitions. By considering all the possible weights in the approximation
problem, they defined the class of \emph{weighted Shapley values}.

\item Hammer and Holzman~\cite{HamHol92} considered the problem of approximating a pseudo-Boolean function by another pseudo-Boolean function of
smaller degree in the sense of standard (non-weighted and non-constrained) least squares. They showed that the Banzhaf power index appears as the coefficients of
the linear terms in the solution of the approximation problem by functions of degree at most one. Later, this problem was generalized by
Grabisch et al.~\cite{GraMarRou00} who showed that the Banzhaf interaction index appears as the leading coefficients of the best least squares
approximations by functions of specified degrees.
\end{itemize}

A natural way to generalize the non-weighted approach of Hammer and Holzman (we recall it in Section 2) consists in adding the following weighted,
probabilistic viewpoint: A weight $w(S)$ is assigned to every coalition $S$ of
players and interpreted as the probability that coalition $S$ forms.\footnote{This probabilistic approach was considered for instance in \cite{ChaGolKeaRou88,Cur87,Owe72}.} On this issue, we note that the weighted least squares problem associated with the probability distribution $w$ was studied in Ding et al.~\cite{DinLaxCheChe08,DinLaxCheChe10} in the special case when the players behave independently of each other to form coalitions.

In Section 3 we briefly recall the setting and main results of the approximation problem considered by Ding et al.~\cite{DinLaxCheChe08,DinLaxCheChe10}. We then introduce a weighted Banzhaf interaction index associated with $w$ by considering, as in Hammer and
Holzman's approach, the leading coefficients of the approximations of specified degrees. We also derive explicit expressions for this index, which allow us to generalize some of the results in \cite{DinLaxCheChe10}.

In Section 4 we investigate the main properties of this new class of indexes. For instance we prove that they define a subclass of the family of so-called probabilistic interaction indexes introduced in Fujimoto et al.~\cite{FujKojMar06}, we analyze their behavior with respect to null players and dummy coalitions, and we describe their symmetric versions.

Finally, in Section 5 we discuss interpretations of the Banzhaf and Shapley interaction indexes as centers of mass of weighted Banzhaf interaction indexes and we introduce an absolute interaction index associated to each weighted Banzhaf interaction index, which allows us to compute the coefficient of determination of the best $k$th approximations.

\section{Interaction indexes}

In this section we recall the concepts of power and interaction indexes introduced in cooperative game theory and how the Banzhaf index
can be obtained from the solution of a standard least squares approximation problem.

Recall that a (\emph{cooperative})~\emph{game} on a finite set of players $N=\{1,\ldots,n\}$ is a set function $v\colon 2^N\to\R$ which assigns
to each coalition $S$ of players a real number $v(S)$ representing the \emph{worth} of $S$.\footnote{Usually, the condition $v(\varnothing)=0$
is required for $v$ to define a game. However, we do not need this restriction in the present paper.} Through the usual identification of the
subsets of $N$ with the elements of $\{0,1\}^n$, a game $v\colon 2^N\to\R$ can be equivalently described by a pseudo-Boolean function
$f\colon\{0,1\}^n\to\R$. The correspondence is given by $v(S)=f(\mathbf{1}_S)$ and
\begin{equation}\label{eq:pBfPF}
f(\bfx)=\sum_{S\subseteq N} v(S)\,\prod_{i\in S}x_i\,\prod_{i\in N\setminus S}(1-x_i).
\end{equation}
To avoid cumbersome notation, we will henceforth use the same symbol to denote both a given pseudo-Boolean function and its underlying set
function (game), thus writing $f\colon\{0,1\}^n\to\R$ or $f\colon 2^N\to\R$ indifferently.

Equation~(\ref{eq:pBfPF}) shows that any pseudo-Boolean function $f\colon\{0,1\}^n\to\R$ can always be represented by a multilinear polynomial
of degree at most $n$ (see \cite{HamRud68}), which can be further simplified into
\begin{equation}\label{eq:fMob}
f(\bfx)=\sum_{S\subseteq N} a(S)\,\prod_{i\in S}x_i\, ,
\end{equation}
where the set function $a\colon 2^N\to\R$, called the \emph{M\"obius transform} of $f$, is defined by
$$
a(S)=\sum_{T\subseteq S} (-1)^{|S|-|T|}\, f(T).
$$

Let $\mathcal{G}^N$ denote the set of games on $N$. A \emph{power index} \cite{Sha53} on $N$ is a function $\phi\colon\mathcal{G}^N\times
N\to\R$ that assigns to every player $i\in N$ in a game $f\in\mathcal{G}^N$ his/her prospect $\phi(f,i)$ from playing the game.
An \emph{interaction index} \cite{GraRou99} on $N$ is a function $I\colon\mathcal{G}^N\times 2^N\to\R$ that measures in a game
$f\in\mathcal{G}^N$ the interaction degree among the players of a coalition $S\subseteq N$.

For instance, the \emph{Banzhaf interaction index} \cite{GraRou99} of a coalition $S\subseteq N$ in a game $f\in\mathcal{G}^N$ is defined by
$$
I_\mathrm{B}(f,S)=\sum_{T\supseteq S} \Big(\frac 12\Big)^{|T|-|S|} a(T)
$$
and the \emph{Banzhaf power index} \cite{DubSha79} of a player
$i\in N$ in a game $f\in\mathcal{G}^N$ is defined by $\phi_\mathrm{B}(f,i)=I_\mathrm{B}(f,\{i\})$.

It is noteworthy that $I_\mathrm{B}(f,S)$ can be interpreted as an average of the \emph{$S$-difference} (or \emph{discrete $S$-derivative})
$\Delta^Sf$ of $f$. Indeed, it can be shown (see \cite[{\S}2]{GraMarRou00}) that
\begin{equation}\label{DiscBanzhaf}
I_\mathrm{B}(f,S)=\frac 1{2^n}\sum_{\bfx\in\{0,1\}^n}\Delta^Sf(\bfx) = \frac{1}{2^{n-|S|}}\sum_{T\subseteq N\setminus S} (\Delta^Sf)(T),
\end{equation}
where $\Delta^Sf$ is defined inductively by $\Delta^{\varnothing}f=f$ and $\Delta^Sf=\Delta^{\{i\}}\Delta^{S\setminus\{i\}}f$ for $i\in S$, with
$\Delta^{\{i\}}f(\mathbf{x})=f(\mathbf{x}\mid x_i=1)-f(\mathbf{x}\mid x_i=0)$.

By extending formally any pseudo-Boolean function $f$ to $[0,1]^n$ by linear interpolation, we can define the \emph{multilinear extension} of
$f$ (see Owen~\cite{Owe72,Owe88}), that is, the multilinear polynomial $\bar{f}\colon [0,1]^n\to\R$ defined by
$$
\bar{f}(\bfx)=\sum_{S\subseteq N} f(S)\,\prod_{i\in S}x_i\,\prod_{i\in N\setminus S}(1-x_i)=\sum_{S\subseteq N} a(S)\,\prod_{i\in S}x_i\, .
$$

By extending also the concept of $S$-difference to the multilinear polynomials defined on $[0,1]^n$, we also have the following identities
(see~\cite{Owe88})
\begin{equation}\label{eq:we7r89}
I_\mathrm{B}(f,S)=\textstyle{(\Delta^S\bar{f})\big(\boldsymbol{\frac 12}\big)}=\displaystyle{\int_{[0,1]^n}\Delta^S\bar{f}(\bfx)\, d\bfx},
\end{equation}
where $\boldsymbol{\frac 12}$ stands for $\big(\frac 12,\ldots,\frac 12\big)$.

Since the $S$-difference operator $\Delta^S$ has the same effect as the $S$-derivative operator $D^S$ (i.e., the partial derivative operator
with respect to the variables in $S$) when applied to multilinear polynomials defined on $[0,1]^n$, we also have
\begin{equation}\label{eq:we7r89a}
I_\mathrm{B}(f,S)=\textstyle{(D^S\bar{f})\big(\boldsymbol{\frac 12}\big)}=\displaystyle{\int_{[0,1]^n}D^S\bar{f}(\bfx)\, d\bfx}.
\end{equation}

We now recall how the Banzhaf interaction index can be obtained from a least squares approximation problem, as investigated by Hammer and
Holzman~\cite{HamHol92} and Grabisch et al.~\cite{GraMarRou00}. For $k\in\{0,\ldots,n\}$, denote by $V_k$ the set of all multilinear polynomials
$g\colon\{0,1\}^n\to\R$ of degree at most $k$, that is of the form
$$
g(\bfx)=\sum_{\textstyle{S\subseteq N\atop |S|\leqslant k}} c(S)\prod_{i\in S}x_i\, ,
$$
where the coefficients $c(S)$ are real numbers. For a given pseudo-Boolean function $f\colon\{0,1\}^n\to\R$, the best $k$th approximation of $f$
is the unique multilinear polynomial $f_k\in V_k$ that minimizes the squared distance
\begin{equation}\label{eq:NonWeiDist}
\sum_{\bfx\in\{0,1\}^n}\big(f(\bfx)-g(\bfx)\big)^2=\sum_{T\subseteq N}\big(f(T)-g(T)\big)^2
\end{equation}
among all functions $g\in V_k$. A closed-form expression of $f_k$ was given in \cite{HamHol92} for $k=1$ and $k=2$ and in \cite{GraMarRou00} for
arbitrary $k\leqslant n$. In fact, when $f$ is given in its multilinear form (\ref{eq:fMob}) we obtain
\begin{equation}\label{eq:fkMob}
f_k(\bfx)=\sum_{\textstyle{S\subseteq N\atop |S|\leqslant k}} a_k(S)\prod_{i\in S}x_i,
\end{equation}
where
\begin{equation}\label{eq:aks1}
a_k(S)=a(S)+(-1)^{k-|S|}\sum_{\textstyle{T\supseteq S\atop |T|>k}}{|T|-|S|-1\choose k-|S|}\,\Big(\frac 12\Big)^{|T|-|S|}a(T).
\end{equation}
It is then easy to see that
\begin{equation}\label{shgfsdf}
I_\mathrm{B}(f,S)=a_{|S|}(S).
\end{equation}
Thus $I_\mathrm{B}(f,S)$ is exactly the coefficient of the monomial $\prod_{i\in S}x_i$ in the best approximation of $f$ by a multilinear
polynomial of degree at most $|S|$.

\section{Weighted Banzhaf interaction indexes}

The approximation problem described in the previous section uses the standard (non-weighted) Euclidean distance (\ref{eq:NonWeiDist}), for
which all the subsets (or coalitions of players) are considered on the same footing. Now, suppose that some coalitions are more important than
some others, for instance because they are more likely to form. To take these importances into consideration, it is natural to generalize the approximation problem by considering an appropriate weighted Euclidean distance. Thus modified, this approximation problem will then allow us to define a concept of weighted Banzhaf interaction index.

This weighted approximation problem was actually presented and solved (under the independence assumption) in Ding et al.~\cite{DinLaxCheChe08,DinLaxCheChe10}. We now briefly recall the setting of this problem as well as some of the most relevant results.

Given a weight function $w\colon\{0,1\}^n\to\left]0,\infty\right[$ and a pseudo-Boolean function $f\colon\{0,1\}^n\to\R$, we define
the \emph{best $k$th approximation of $f$} as the unique multilinear polynomial $f_k\in V_k$ that minimizes the squared distance
\begin{equation}\label{eq:WeiDist}
\sum_{\bfx\in\{0,1\}^n}w(\bfx)\big(f(\bfx)-g(\bfx)\big)^2=\sum_{S\subseteq N}w(S)\big(f(S)-g(S)\big)^2
\end{equation}
among all functions $g\in V_k$.

Clearly, we can assume without loss of generality that the weights $w(S)$ are (multiplicatively) normalized so that $\sum_{S\subseteq N}w(S)=1$. We then immediately see that the weights define a probability distribution over $2^N$ and we can interpret $w(S)$ as the probability that coalition $S$ forms, that is, $w(S) =
\Pr(C=S)$, where $C$ denotes a random coalition.

Now, suppose that the players behave independently of each other to form coalitions, which means that the events $(C\ni i)$, for $i\in N$, are
independent. In this case, also the indicator random variables $X_i=\mathrm{Ind}(C\ni i)$, for $i\in N$, are independent. Setting $p_i=\Pr(C\ni
i)=\sum_{S\ni i}w(S)$, we then have $p_i=\Pr[X_i=1]=E[X_i]$, $0<p_i<1$, and
$$
w(S) = \prod_{i\in S}p_i\prod_{i\in N\setminus S}(1-p_i)
$$
or, equivalently,
$$
w(\bfx)=\prod_{i\in N}p_i^{x_i}(1-p_i)^{1-x_i}.
$$

\begin{remark}\label{rem:MLE757}
This interpretation of $w(T)$ as a probability is precisely the one proposed by Owen~\cite{Owe88} in the interpretation of the multilinear
extension of a game as an expected value: Given a game $f\colon 2^N\to\R$, we have
$$
\bar{f}(p_1,\ldots,p_n)=\sum_{S\subseteq N}w(S)f(S)=E[f(C)],
$$
where $C$ is a random coalition.
\end{remark}

The set $V_k$ is clearly a linear space of dimension $\sum_{s=0}^k {n\choose s}$ spanned by the basis $B_k=\{u_S : S\subseteq N, \, |S|\leqslant k\}$, where the functions
$u_S\colon \{0,1\}^n\to\R$ (called unanimity games in game theory) are defined by $u_S(\bfx)=\prod_{i\in S}x_i$. Note that the distance defined in (\ref{eq:WeiDist}) is the natural $L^2$-distance associated with the measure $w$ and corresponds to the weighted Euclidean inner product
$$
\langle f,g\rangle=\sum_{\bfx\in\{0,1\}^n}w(\bfx)f(\bfx)g(\bfx).
$$
Thus the solution of this approximation problem exists and is uniquely determined by the orthogonal projection of $f$ onto $V_k$. This
projection can be easily expressed in any orthonormal basis of $V_k$. In this respect, it was shown in ~\cite{DinLaxCheChe10} that the set
$B'_k=\{v_S : S\subseteq N, \, |S|\leqslant k\}$, where $v_S\colon\{0,1\}^n\to\R$ is given by
$$
v_S(\bfx)=\prod_{i\in S}\frac{x_i-p_i}{\sqrt{p_i(1-p_i)}}=\sum_{T\subseteq S}\frac{\prod_{i\in S\setminus T}(-p_i)}{\prod_{i\in
S}\sqrt{p_i(1-p_i)}}\, u_T(\bfx)
$$
forms such an orthonormal basis for $V_k$.

The following immediate theorem gives the components of the best $k$th approximation of a pseudo-Boolean function $f\colon\{0,1\}^n\to\R$ in the
basis $B'_k$.
\begin{theorem}\label{thm:Coeffk}\cite[Theorem 4]{DinLaxCheChe10}
The best $k$th approximation of $f\colon\{0,1\}^n\to\R$ is the function
\begin{equation}\label{eq:akS2}
f_k=\sum_{\textstyle{T\subseteq N\atop |T|\leqslant k}}\langle f,v_T\rangle \, v_T\, .
\end{equation}
\end{theorem}

By expressing the functions $v_T$ in the basis $B_k$, we immediately obtain the following expression of $f_k$ in terms of the functions $u_S$:
\begin{equation}\label{eq:akS2bis}
f_k= \sum_{\textstyle{S\subseteq N\atop |S|\leqslant k}}
a_k(S)\, u_S\, ,
\end{equation}
where
\begin{equation}\label{eq:akS}
a_k(S)=\sum_{\textstyle{T\supseteq S\atop |T|\leqslant k}} \frac{\prod_{i\in T\setminus S}(-p_i)}{\prod_{i\in T}\sqrt{p_i(1-p_i)}} \, \langle
f,v_T\rangle.
\end{equation}

Let $\mathbf{p}$ stand for $(p_1,\ldots,p_n)$. By analogy with (\ref{shgfsdf}), in order to measure the interaction degree among players in a
game $f\colon\{0,1\}^n\to\R$, we naturally define an index $I_{\mathrm{B},\mathbf{p}}\colon\mathcal{G}^N\times 2^N\to\R$ as
$I_{\mathrm{B},\mathbf{p}}(f,S)=a_{|S|}(S)$, where $a_{|S|}(S)$ is obtained from $f$ by (\ref{eq:akS}). We will see in the next section that
this index indeed measures a power degree when $|S|=1$ and an interaction degree when $|S|\geqslant 2$.

\begin{definition}
Let $I_{\mathrm{B},\mathbf{p}}\colon\mathcal{G}^N\times 2^N\to\R$ be defined as
$$
I_{\mathrm{B},\mathbf{p}}(f,S)= \frac{\langle f,v_S\rangle}{\prod_{i\in S}\sqrt{p_i(1-p_i)}}\, ,
$$
that is,
\begin{equation}\label{eq:GenBI1}
I_{\mathrm{B},\mathbf{p}}(f,S)= \frac{1}{\prod_{i\in S}p_i(1-p_i)}\sum_{\bfx\in\{0,1\}^n}w(\bfx)f(\bfx)\prod_{i\in S}(x_i-p_i).
\end{equation}
\end{definition}

Clearly, formula (\ref{eq:GenBI1}) can be immediately rewritten as a sum over subsets as follows:
\begin{equation}\label{eq:GenBI1a}
I_{\mathrm{B},\mathbf{p}}(f,S)=\sum_{T\subseteq N}(-1)^{|S\setminus T|}\, f(T)\,\prod_{i\in T\setminus
S}p_i\prod_{i\in N\setminus(T\cup S)}(1-p_i).
\end{equation}

\begin{remark}
The definition of the index $I_{\mathrm{B},\mathbf{p}}$ is close to that of the transformation $T$ considered in Ding et al.~\cite{DinLaxCheChe10}, where the components of $T(f)$ are defined by $\alpha_S(f)=\langle f,v_S\rangle$. However, our approach (which is closer to Hammer and Holzman's~\cite{HamHol92}) is not only equivalent to Ding et al.'s but leads to easier interpretations and computations as will be shown in the next paragraphs.
\end{remark}

We have defined an interaction index from an approximation (projection) problem. Conversely, this index characterizes this approximation
problem. Indeed, as the following result shows, the best $k$th approximation of $f\colon\{0,1\}^n\to\R$ is the unique function of $V_k$ that
preserves the interaction index for all the $s$-subsets such that $s\leqslant k$. The non-weighted analogue of this result was established in
\cite{GraMarRou00} for the Banzhaf interaction index $I_{\mathrm{B}}$.

\begin{proposition}\label{prop:Proj-IB}
A function $f_k\in V_k$ is the best $k$th approximation of $f\colon\{0,1\}^n\to\R$ if and only if
$I_{\mathrm{B},\mathbf{p}}(f,S)=I_{\mathrm{B},\mathbf{p}}(f_k,S)$ for all $S\subseteq N$ such that $|S|\leqslant k$.
\end{proposition}

\begin{proof}
By definition, we have $I_{\mathrm{B},\mathbf{p}}(f,S)=I_{\mathrm{B},\mathbf{p}}(f_k,S)$ if and only if $\langle f-f_k,v_S\rangle =0$ for all
$S\subseteq N$ such that $|S|\leqslant k$, and the latter condition characterizes the projection of $f$ onto $V_k$.
\end{proof}

Since the best $n$th approximation of $f$ is $f$ itself, by (\ref{eq:akS2}) we immediately see that $f$ can be expressed in terms of
$I_{\mathrm{B},\mathbf{p}}$ as
\begin{equation}\label{eq:Approx}
f(\bfx) = \sum_{T\subseteq N} I_{\mathrm{B},\mathbf{p}}(f,T)\, \prod_{i\in T}(x_i-p_i),
\end{equation}
which shows that the map $f\mapsto\{I_{\mathrm{B},\mathbf{p}}(f,S):S\subseteq N\}$ is a linear bijection.

We also have the following representation result, which generalizes the first equalities in (\ref{eq:we7r89}) and (\ref{eq:we7r89a}).

\begin{proposition}\label{prop:d4546dg}
For every $f\colon\{0,1\}^n\to\R$ and every $S\subseteq N$, we have
\begin{equation}\label{eq:IB-DSk56}
I_{\mathrm{B},\mathbf{p}}(f,S)=(D^S \bar{f})(\mathbf{p})=(\Delta^S \bar{f})(\mathbf{p}).
\end{equation}
In particular, $I_{\mathrm{B},\mathbf{p}}(f,\varnothing)= \bar{f}(\mathbf{p}) =\sum_{\bfx\in\{0,1\}^n}w(\bfx)f(\bfx)$.
\end{proposition}

\begin{proof}
The result immediately follows from comparing (\ref{eq:Approx}) with the Taylor expansion of $\bar f$ at $\mathbf{p}$. The particular case was
discussed in Remark~\ref{rem:MLE757}.
\end{proof}

\begin{example}
Consider the 3-person majority game defined by
$$
f(x_1,x_2,x_3)=x_1x_2+x_2x_3+x_3x_1-2x_1x_2x_3.
$$
By (\ref{eq:we7r89}) and (\ref{eq:IB-DSk56}), we have $I_{\mathrm{B}}(f,\{i,j\})=0$ and $I_{\mathrm{B},\mathbf{p}}(f,\{i,j\})=1-2p_k$, where $\{i,j,k\}=\{1,2,3\}$. Intuitively, if $p_k$ is close to $1$, then the coalitions containing $k$ are most likely to form. In these coalitions, the presence of only one of the remaining players is sufficient to form a winning coalition, thus explaining the negative interaction between $i$ and $j$. A similar conclusion can be drawn if $p_k$ is close to $0$.
\end{example}

Explicit conversion formulas between the interaction index and the best approximation can be easily derived from the preceding results. On
the one hand, by (\ref{eq:akS}), we have
\begin{equation}\label{eq:sdfd54}
a_k(S)=\sum_{\textstyle{T\supseteq S\atop |T|\leqslant k}}I_{\mathrm{B},\mathbf{p}}(f,T)\, \prod_{i\in T\setminus S}(-p_i)\, ,\qquad
\mbox{for}~|S|\leqslant k.
\end{equation}
On the other hand, by Propositions~\ref{prop:Proj-IB} and \ref{prop:d4546dg} and Equation (\ref{eq:akS2bis}), we also have
\begin{eqnarray*}
I_{\mathrm{B},\mathbf{p}}(f,S) &=& I_{\mathrm{B},\mathbf{p}}(f_k,S) ~=~ (\Delta^S \bar{f}_k)(\mathbf{p})\\
&=& \sum_{\textstyle{T\subseteq N\atop |T|\leqslant k}}a_k(T)\,(\Delta^S \bar{u}_T)(\mathbf{p})\, ,
\end{eqnarray*}
that is, since $\Delta^S \bar{u}_T=\bar{u}_{T\setminus S}$ if $S\subseteq T$ and $0$ otherwise,
\begin{equation}\label{eq:sfhksd45}
I_{\mathrm{B},\mathbf{p}}(f,S)=\sum_{\textstyle{T\supseteq S\atop |T|\leqslant k}}a_k(T)\, \prod_{i\in T\setminus S}p_i\, ,\qquad
\mbox{for}~|S|\leqslant k.
\end{equation}
Taking $k=n$ in (\ref{eq:sfhksd45}), we immediately derive the following expression of $I_{\mathrm{B},\mathbf{p}}(f,S)$ in terms of the M\"obius
transform of $f$:
\begin{equation}\label{eq:sfhfdgfksd45}
I_{\mathrm{B},\mathbf{p}}(f,S)=\sum_{T\supseteq S}a(T)\, \prod_{i\in T\setminus S}p_i\, .
\end{equation}

Combining formulas (\ref{eq:sdfd54}) and (\ref{eq:sfhfdgfksd45}) allows us to express the coefficients $a_k(S)$ explicitly in terms of the M\"obius transform of $f$. We give this expression in the following proposition, which generalizes (\ref{eq:aks1}) and \cite[Theorem 7]{DinLaxCheChe10}.

\begin{proposition}
The best $k$th approximation of $f\colon\{0,1\}^n\to\R$ is given by (\ref{eq:fkMob}), where
\[
a_k(S)=a(S)+(-1)^{k-|S|}\sum_{\textstyle{T\supseteq S\atop |T|> k}}{|T|-|S|-1\choose k-|S|} \bigg(\prod_{i\in T\setminus S}p_i\bigg)\, a(T)\, ,\qquad\mbox{for}~|S|\leqslant k.
\]
\end{proposition}

\begin{proof}
By combining (\ref{eq:sdfd54}) and (\ref{eq:sfhfdgfksd45}) and then permuting the sums, we obtain
\[
a_k(S)=\sum_{T\supseteq S}\,\bigg(\prod_{i\in T\setminus S}p_i\bigg)\, a(T) \sum_{\textstyle{R:S\subseteq R\subseteq T\atop |R|\leqslant k}}\,  (-1)^{|R|-|S|},
\]
where the explicit computation of the inner sum was done in \cite[p.~20]{DinLaxCheChe10}.
\end{proof}

It is important to remember that the special case $\mathbf{p}=\boldsymbol{\frac 12}$ corresponds to the non-weighted approximation
problem investigated first by Hammer and Holzman and for which the index $I_{\mathrm{B},\mathbf{p}}$ reduces to the Banzhaf interaction index
$I_{\mathrm{B}}$. For this reason, we will call the index $I_{\mathrm{B},\mathbf{p}}$ the \emph{weighted Banzhaf interaction index}. Its
expressions in (\ref{eq:GenBI1}) and (\ref{eq:GenBI1a}) provide the following alternative formulas for the Banzhaf interaction index. The second one was found in \cite[Table 3]{GraMarRou00}.

\begin{corollary}
For every $f\colon\{0,1\}^n\to\R$ and every $S\subseteq N$, we have
$$
I_{\mathrm{B}}(f,S)=\frac{1}{2^{n-|S|}}\sum_{\bfx\in \{0,1\}^n}f(\bfx)\prod_{i\in S}(2 x_i-1)=\frac{1}{2^{n-|S|}}\sum_{T\subseteq
N}(-1)^{|S\setminus T|}f(T).
$$
\end{corollary}

\section{Properties and interpretations}

Most of the interaction indexes defined for games, including the Banzhaf interaction index, share a set of fundamental properties such as
linearity, symmetry, and monotonicity (see \cite{FujKojMar06}). Many of them can also be expressed as expected values of the discrete
derivatives (differences) of their arguments (see for instance (\ref{DiscBanzhaf})). In this section we show that the index
$I_{\mathrm{B},\mathbf{p}}$ fulfills many of these properties.

The first result follows from the very definition of the index.

\begin{proposition}\label{prop:lin56}
For every $S\subseteq N$, the mapping $f\mapsto I_{\mathrm{B},\mathbf{p}}(f,S)$ is linear.
\end{proposition}

We now provide an interpretation of $I_{\mathrm{B},\mathbf{p}}(f,S)$ as an expected value of the $S$-difference $\Delta^Sf$ of $f$. This
interpretation is a direct generalization of the one obtained for the Banzhaf index $I_{\mathrm{B}}$; see formula (\ref{DiscBanzhaf}). The proof
immediately follows from Proposition~\ref{prop:d4546dg} and thus is omitted.

\begin{proposition}\label{prop:d4546dg56}
For every $f\colon\{0,1\}^n\to\R$ and every $S\subseteq N$, we have
\begin{equation}\label{eq:IB-DSk56fsf}
I_{\mathrm{B},\mathbf{p}}(f,S)=\sum_{\bfx\in\{0,1\}^n}w(\bfx)\,\Delta^S f(\bfx).
\end{equation}
\end{proposition}

Rewriting (\ref{eq:IB-DSk56fsf}) as a sum over subsets, we obtain
\begin{equation}\label{eq:s1f231}
I_{\mathrm{B},\mathbf{p}}(f,S)=\sum_{T\subseteq N}w(T)\,(\Delta^S f)(T)=E[(\Delta^S f)(C)],
\end{equation}
where $C$ denotes a random coalition. Notice that formula (\ref{eq:s1f231}) can also be obtained from (\ref{eq:sfhfdgfksd45}) by using the
random indicator vector $\mathbf{X}=(X_1,\ldots,X_n)$. Indeed, we have
$$
I_{\mathrm{B},\mathbf{p}}(f,S)=\sum_{T\supseteq S} a(T)\, E\bigg[\prod_{i\in T\setminus S}X_i\bigg]=E[\Delta^S f(\mathbf{X})].
$$

\begin{remark}
By combining Propositions~\ref{prop:Proj-IB} and \ref{prop:d4546dg56}, we see that the best $k$th approximation of $f$ is the unique multilinear
polynomial of degree at most $k$ that agrees with $f$ in all average $S$-differences for $|S|\leqslant k$.
\end{remark}

Since $(\Delta^S f)(T)=(\Delta^S f)(T\setminus S)$, we can actually rewrite the sum in (\ref{eq:s1f231}) as a sum over the subsets of
$N\setminus S$. We then obtain the following result, which also generalizes (\ref{DiscBanzhaf}).

\begin{theorem}\label{thm:main654}
For every $f\colon\{0,1\}^n\to\R$ and every $S\subseteq N$, we have
\begin{equation}\label{eq:1df2g1df23g}
I_{\mathrm{B},\mathbf{p}}(f,S)=\sum_{T\subseteq N\setminus S} p_T^S\,(\Delta^Sf)(T),
\end{equation}
where $p_T^S=\Pr(T\subseteq C\subseteq S\cup T)=\prod_{i\in T}p_i\prod_{i\in(N\setminus S)\setminus T}(1-p_i)$. Moreover, we have
\begin{equation}\label{eq:1df2df3g}
\sum_{T\subseteq N\setminus S}\,p_T^S = 1.
\end{equation}
\end{theorem}

\begin{proof}
Partitioning $T\subseteq N$ into $K\subseteq N\setminus S$ and $L\subseteq S$, we can rewrite the sum in (\ref{eq:s1f231}) as
$$
I_{\mathrm{B},\mathbf{p}}(f,S) = \sum_{K\subseteq N\setminus S}(\Delta^S f)(K)\,\sum_{L\subseteq S}w(K\cup L)
$$
where the inner sum is exactly $\Pr(K\subseteq C\subseteq K\cup S)$. Moreover, we have
\begin{eqnarray*}
\lefteqn{\Pr(K\subseteq C\subseteq K\cup S)}\\
&=& \Pr(X_i=1~\forall i\in K~\mbox{and}~X_i=0~\forall i\in (N\setminus S)\setminus K)\\
&=& E\bigg[\prod_{i\in K}X_i\prod_{i\in(N\setminus S)\setminus K}(1-X_i)\bigg]~=~ \prod_{i\in K}p_i\prod_{i\in(N\setminus S)\setminus K}(1-p_i),
\end{eqnarray*}
which proves the first part of the theorem. For the second part, we simply apply (\ref{eq:1df2g1df23g}) to $f=u_S$ to obtain $\sum_{T\subseteq
N\setminus S}\,p_T^S = I_{\mathrm{B},\mathbf{p}}(u_S,S) = 1$.
\end{proof}

\begin{remark}
When $S$ is a singleton, $S=\{i\}$, from (\ref{eq:1df2g1df23g}) we derive the following explicit expression for the weighted Banzhaf power index
$$
I_{\mathrm{B},\mathbf{p}}(f,\{i\})=\sum_{T\subseteq N\setminus\{i\}} \big(w(T)+w(T\cup\{i\})\big)\,\big(f(T\cup\{i\})-f(T)\big).
$$
\end{remark}

Interaction indexes of the form (\ref{eq:1df2g1df23g}) with nonnegative coefficients satisfying property (\ref{eq:1df2df3g}) are called
\emph{probabilistic interaction indexes} (see \cite{FujKojMar06}). These indexes share the following probabilistic interpretation. Suppose that
any coalition $S\subseteq N$ joins a coalition $T\subseteq N\setminus S$ at random with (subjective) probability $p_T^S$. Then the right-hand
side in (\ref{eq:1df2g1df23g}) is simply the expected value of the \emph{marginal interaction}  $(\Delta^Sf)(T)$ (called \emph{marginal
contribution}, if $|S|=1$); see also \cite[{\S}2]{GraMarRou00}.

In the case of the index $I_{\mathrm{B},\mathbf{p}}$, we have the following additional interpretations of $p_T^S$ as conditional probabilities. The proof is straightforward and hence omitted.

\begin{proposition}
For every $S\subseteq N$ and every $T\subseteq N\setminus S$, the coefficient $p_T^S$ defined in (\ref{eq:1df2g1df23g}) satisfies
$$
p_T^S=\Pr(C=S\cup T\mid C\supseteq S)=\Pr(C=T\mid C\subseteq N\setminus S),
$$
where $C$ denotes a random coalition.
\end{proposition}

In terms of the multilinear extension $\bar{f}$ of $f$, we also have the following interpretation of $I_{\mathrm{B},\mathbf{p}}$, which
generalizes the second equalities in (\ref{eq:we7r89}) and (\ref{eq:we7r89a}).
\begin{proposition}\label{prop:jhfh43}
Let $F_1,\ldots,F_n$ be cumulative distribution functions on $[0,1]$. Then
\begin{equation}\label{eq:safda987}
I_{\mathrm{B},\mathbf{p}}(f,S)=\int_{[0,1]^n}(\Delta^S\bar{f})(\mathbf{x})\, dF_1(x_1)\cdots dF_n(x_n)
\end{equation}
for every $f\colon\{0,1\}^n\to\R$ and every $S\subseteq N$ if and only if $p_i=\int_0^1x\, dF_i(x)$ for every $i\in N$.
\end{proposition}

\begin{proof}
By linearity of the index, Equation~(\ref{eq:safda987}) holds for every $f\colon\{0,1\}^n\to\R$ and every $S\subseteq N$ if and only if it holds
for every $f=u_T$, with $T\subseteq N$, and every $S\subseteq N$. Thus this condition is equivalent to
\begin{equation}\label{eq:safsd87}
\prod_{i\in T\setminus S}p_i =\int_{[0,1]^n}\prod_{i\in T\setminus S}x_i\, dF_1(x_1)\cdots dF_n(x_n)
\end{equation}
for every $T\subseteq N$ and every $S\subseteq T$. The result then immediately follows since the right-hand integral in (\ref{eq:safsd87})
reduces to $\prod_{i\in T\setminus S}\int_0^1x_i\, dF_i(x_i)$.
\end{proof}

\begin{remark}
Clearly, the functions $F_1,\ldots,F_n$ in Proposition~\ref{prop:jhfh43} are not uniquely determined by $\mathbf{p}$. For instance, we could choose the power function
$F_i(x)=x^{p_i/(1-p_i)}$ or the one-step function $F_i(x)=\chi_{[p_i,1]}$. We could as well consider the beta distribution with parameters $p_i$
and $1-p_i$.
\end{remark}

We now analyze the behavior of the interaction index $I_{\mathrm{B},\mathbf{p}}$ on some special classes of functions. We continue to identify
pseudo-Boolean functions on $\{0,1\}^n$ with games on $N$ and vice versa.

Recall that a \emph{null player} in a game $f\in\mathcal{G}^N$ is a player $i\in N$ such that $f(T\cup\{i\})=f(T)$ for all $T\subseteq
N\setminus\{i\}$. Equivalently, we have $\Delta^{\{i\}}f(\bfx)=0$ for all $\bfx\in\{0,1\}^n$ and the variable $x_i$ is said to be
\emph{ineffective} for $f$. In this case, we have
$$
f(\bfx)=\sum_{T\subseteq N\setminus\{i\}}(\Delta^Tf)(\mathbf{0})\,\prod_{j\in T}x_j=\sum_{T\subseteq N\setminus\{i\}}a(T)\,\prod_{j\in T}x_j\, ,
$$
where $\mathbf{0}=(0,\ldots,0)$.

Define $I_f=\{i\in N:\mbox{$x_i$ ineffective for $f$}\}$; that is, $I_f$ is the set of null players in $f$. From either (\ref{eq:IB-DSk56fsf}),
(\ref{eq:s1f231}), or (\ref{eq:1df2g1df23g}), we immediately derive the following result, which states that any coalition containing at least
one null player in $f$ has necessarily a zero interaction.

\begin{proposition}\label{prop:jjdsl44}
For every $f\colon\{0,1\}^n\to\R$ and every $S\subseteq N$ such that $S\cap I_f\neq\varnothing$, we have $I_{\mathrm{B},\mathbf{p}}(f,S)=0$.
\end{proposition}

Recall also that a \emph{dummy player} in a game $f\in\mathcal{G}^N$ is a player $i\in N$ such that $f(T\cup\{i\})=f(T)+f(\{i\})-f(\varnothing)$
for all $T\subseteq N\setminus\{i\}$. 
We say that a coalition $S\subseteq N$ is \emph{dummy} in $f\in\mathcal{G}^N$ if $f(R\cup T)=f(R)+f(T)-f(\varnothing)$ for every $R\subseteq S$
and every $T\subseteq N\setminus S$. Thus a coalition $S$ and its complement $N\setminus S$ are simultaneously dummy in any game
$f\in\mathcal{G}^N$.

The following proposition gives an immediate interpretation of this definition.

\begin{proposition}\label{prop:dsfa89}
A coalition $S\subseteq N$ is dummy in a game $f\in\mathcal{G}^N$ if and only if there exist games $f_S,f_{N\setminus S}\in\mathcal{G}^N$ such
that $I_{f_S}\supseteq N\setminus S$, $I_{f_{N\setminus S}}\supseteq S$ and $f=f_S+f_{N\setminus S}$.
\end{proposition}

\begin{proof}
For the necessity, just set $f_S(T)=f(T\cap S)$ and $f_{N\setminus S}(T)=f(T\setminus S)-f(\varnothing)$. The sufficiency can be checked
directly.
\end{proof}

Thus Proposition~\ref{prop:dsfa89} states that a coalition $S\subseteq N$ is dummy in $f\in\mathcal{G}^N$ if and only if $f$ is of the form
$$
f(\bfx)=\sum_{T\subseteq S}a(T)\,\prod_{i\in T}x_i+\sum_{\textstyle{T\subseteq N\setminus S\atop T\neq\varnothing}}a(T)\,\prod_{i\in T}x_i\, .
$$

The following result expresses the natural idea that the interaction for coalitions that are properly partitioned by a dummy coalition must be
zero. It is an immediate consequence of Propositions~\ref{prop:lin56}, \ref{prop:jjdsl44}, and \ref{prop:dsfa89}.

\begin{proposition}\label{prop:7dsad7f}
If a coalition $S\subseteq N$ is dummy in a game $f\in\mathcal{G}^N$, then for every coalition $K\subseteq N$ such that $K\cap S\neq\varnothing$
and $K\setminus S\neq\varnothing$, we have $I_{\mathrm{B},\mathbf{p}}(f,K)=0$.
\end{proposition}

We also have the following result, which immediately follows from Proposition~\ref{prop:d4546dg56}.

\begin{proposition}\label{prop:dafdsf2}
If $f\colon\{0,1\}^n\to\R$ is $S$-increasing for some $S\subseteq N$ (i.e., $\Delta^Sf(\bfx)\geqslant 0$ for all $\bfx\in\{0,1\}^n$), then
$I_{\mathrm{B},\mathbf{p}}(f,S)\geqslant 0$.
\end{proposition}

We end this section by describing the weighted Banzhaf interaction indexes that are symmetric. An interaction index $I_{\mathrm{B},\mathbf{p}}$ is said to be \emph{symmetric} (see \cite{GraRou99}) if
$I_{\mathrm{B},\mathbf{p}}(\pi(f),\pi(S))=I_{\mathrm{B},\mathbf{p}}(f,S)$ for every function $f\colon\{0,1\}^n\to\R$, every subset $S\subseteq
N$, and every permutation $\pi$ on $N$, where $\pi(f)$ denotes the function defined by $\pi(f)(x_1,\ldots,x_n)=f(x_{\pi(1)},\ldots,x_{\pi(n)})$.

\begin{proposition}\label{prop:fas8d7}
The index $I_{\mathrm{B},\mathbf{p}}$ is symmetric if and only if the function $w$ is symmetric (i.e., $p_1=\cdots = p_n$).
\end{proposition}

\begin{proof}
If $w$ is symmetric, then the coefficients $p_T^S$ in (\ref{eq:1df2g1df23g}) depend only on $\mathbf{p}$, $|T|$, and $|S|$. Therefore the index
$I_{\mathrm{B},\mathbf{p}}$ is a cardinal-probabilistic index (see \cite{FujKojMar06}), which is symmetric. Conversely, if
$I_{\mathrm{B},\mathbf{p}}$ is symmetric, then, by (\ref{eq:IB-DSk56}), we have
$$
p_i=I_{\mathrm{B},\mathbf{p}}(u_{\{i,j\}},\{j\})=I_{\mathrm{B},\mathbf{p}}(u_{\{i,j\}},\{i\})=p_j
$$
for every $i,j\in N$, $i\neq j$, and hence $w$ is also symmetric.
\end{proof}


By Proposition~\ref{prop:fas8d7}, we immediately see that the Banzhaf interaction index $I_{\mathbf{B}}=I_{\mathbf{B},\boldsymbol{1/2}}$ is
symmetric. Considering the limiting case $\mathbf{p}=\mathbf{0}$, we also see that the M\"obius transform of $f$ (i.e., $a=I_{\mathbf{B},\boldsymbol{0}}$)
can be regarded as a symmetric weighted Banzhaf interaction index.

\section{Related indexes}

In this final section, we establish interesting links between the weighted Banzhaf interaction index and the Banzhaf and Shapley interaction indexes, which provide new interpretations of the latter indexes. We also introduce a normalized version
of the weighted Banzhaf index to compare interactions from different functions (games) and to compute the coefficient of determination of the best $k$th approximations.
\subsection{Links with the Banzhaf and Shapley indexes}
Since the mapping $f\mapsto I_{\mathrm{B},\mathbf{p}}(f,\cdot)$ is a bijection, we can find conversion formulas between $f$, its M\"obius
transform $a$, and $I_{\mathrm{B},\mathbf{p}}(f,\cdot)$.

The conversion from $a$ to $I_{\mathrm{B},\mathbf{p}}(f,\cdot)$ is given in (\ref{eq:sfhfdgfksd45}). From (\ref{eq:sdfd54}), we immediately
obtain the conversion from $I_{\mathrm{B},\mathbf{p}}(f,\cdot)$ to $a$, namely
\begin{equation}\label{eq:s1f3sd}
a(S)=\sum_{T\supseteq S}I_{\mathrm{B},\mathbf{p}}(f,T)\, \prod_{i\in T\setminus S}(-p_i).
\end{equation}

By combining (\ref{eq:sfhfdgfksd45}) and (\ref{eq:s1f3sd}), we easily obtain a conversion
formula from $I_{\mathrm{B},\mathbf{p}}(f,\cdot)$ to $I_{\mathrm{B},\mathbf{p}'}(f,\cdot)$ for every $\mathbf{p}'\in\left]0,1\right[^n$, namely
\begin{equation}\label{eq:dff4df}
I_{\mathrm{B},\mathbf{p}'}(f,S) = \sum_{T\supseteq S}I_{\mathrm{B},\mathbf{p}}(f,T)\, \prod_{i\in T\setminus S}(p'_i-p_i)\, .
\end{equation}

Now, as already discussed, the index $I_{\mathrm{B}}$ can also be expressed in terms of $I_{\mathrm{B},\mathbf{p}}$ simply by setting
$\mathbf{p}=\boldsymbol{\frac 12}$. However, combining (\ref{eq:we7r89}) with (\ref{eq:IB-DSk56}), we also obtain the following alternative
expression
\begin{equation}\label{eq:dfa89}
I_{\mathrm{B}}(f,S)=\int_{[0,1]^n}I_{\mathrm{B},\mathbf{p}}(f,S) \, d\mathbf{p}.
\end{equation}
Equation (\ref{eq:dfa89}) can be interpreted as follows. Suppose that the players behave independently of each other to form coalitions, each
player $i$ with probability $p_i\in\left]0,1\right[$, but this probability is not known a priori. Then, to define an interaction index, it is
natural to consider the average (center of mass) of the weighted indexes over all possibilities of choosing the probabilities.
Equation~(\ref{eq:dfa89}) shows that we then obtain the Banzhaf interaction index.

The \emph{Shapley interaction index} \cite{GraMarRou00,GraRou99} of a coalition $S\subseteq N$ in a game $f\in\mathcal{G}^N$ is defined by
\begin{equation}\label{eq:IggBI23}
I_\mathrm{Sh}(f,S)=\sum_{T\supseteq S} \frac{a(T)}{|T|-|S|+1} = \int_0^1(\Delta^S\bar{f})(x,\ldots,x)\, dx,
\end{equation}
where the set function $a\colon 2^N\to\R$ is the M\"obius transform of $f$.

Combining (\ref{eq:IB-DSk56}) with (\ref{eq:IggBI23}), we obtain an interesting expression of $I_{\mathrm{Sh}}$ in terms of
$I_{\mathrm{B},\mathbf{p}}$, namely
\begin{equation}\label{eq:dfa891}
I_{\mathrm{Sh}}(f,S)=\int_0^1I_{\mathrm{B},(p,\ldots,p)}(f,S) \, dp.
\end{equation}

Here, the players still behave independently of each other to form coalitions but with the same probability $p$. The integral in
(\ref{eq:dfa891}) simply represents the average value of the weighted indexes over all the possible probabilities.

\begin{remark}
\begin{enumerate}
\item[(a)] Formulas (\ref{eq:sfhfdgfksd45}) and (\ref{eq:s1f3sd}) clearly generalize the conversion formulas between $I_{\mathrm{B}}$ and $a$ given in \cite[p.~175]{GraMarRou00}.

\item[(b)] Expressions of power indexes as integrals similar to (\ref{eq:IggBI23}) and (\ref{eq:dfa891}) were proposed and investigated by Straffin~\cite{Stra88}.

\item[(c)] Every cardinal-probabilistic index \cite{FujKojMar06} can be expressed as an integral of $I_{\mathrm{B},(p,\ldots,p)}$ with respect to some distribution function (see \cite[Theorem~4.4]{FujKojMar06}).
\end{enumerate}
\end{remark}

\subsection{Normalized index and coefficients of determination}

We have seen that the interaction index $I_{\mathrm{B},\mathbf{p}}$ is a linear map. This implies that it cannot be considered as an absolute
interaction index but rather as a relative index constructed to assess and compare interactions for a \emph{given} function.

If we want to compare interactions for \emph{different} functions, we need to consider an absolute (normalized) interaction index. Such an index
can be defined as follows. Considering again $2^N$ as a probability space with respect to the measure $w$, we see that, for a nonempty subset
$S\subseteq N$, the index $I_{\mathrm{B},\mathbf{p}}(f,S)$ is the covariance of the random variables $f$ and $v_S/\prod_{i\in
S}\sqrt{p_i(1-p_i)}$. It is then natural to consider the Pearson correlation coefficient instead of the covariance.

\begin{definition}
The \emph{normalized interaction index} is the mapping $$r\colon \{f\colon\{0,1\}^n\to\R :~\mbox{$f$ is non constant}\}\times
(2^N\setminus\{\varnothing\})\to\R$$ defined by
$$
r(f,S)=\frac{I_{\mathrm{B},\mathbf{p}}(f,S)}{\sigma(f)}\prod_{i\in S}\sqrt{p_i(1-p_i)}=\Big\langle\frac{f-E(f)}{\sigma(f)}\, ,v_S\Big\rangle\, ,
$$
where $E(f)$ and $\sigma(f)$ are the expectation and the standard deviation of $f$, respectively, when $f$ is regarded as a random variable.
\end{definition}

From this definition it follows that $-1\leqslant r(f,S)\leqslant 1$. Moreover, this index remains unchanged under interval scale
transformations, that is, $r(af+b,S)=r(f,S)$ for all $a>0$ and $b\in\R$.

\begin{remark}
By definition of the normalized interaction index, for every nonempty $S\subseteq N$, we have the inequality
$$
|I_{\mathrm{B},\mathbf{p}}(f,S)|\leqslant \frac{\sigma(f)}{\prod_{i\in S}\sqrt{p_i(1-p_i)}}\, .
$$
The equality holds if and only if there exist $a,b\in\R$ such that $f=a\,v_S+b$.
\end{remark}

The normalized index is also useful to compute the \emph{coefficient of determination} $R^2_k(f)=\sigma^2(f_k)/\sigma^2(f)$ of the best $k$th approximation of $f$ (assuming $f$ nonconstant). Since
$E(f_k)=I_{\mathrm{B},\mathbf{p}}(f_k,\varnothing)=I_{\mathrm{B},\mathbf{p}}(f,\varnothing)=E(f)$ (see Proposition~\ref{prop:Proj-IB}), by
(\ref{eq:akS2}), we obtain
\begin{eqnarray*}
R^2_k(f) &=& \frac{1}{\sigma^2(f)}\,\|f_k-E(f_k)\|^2\\
&=& \frac{1}{\sigma^2(f)}\sum_{\textstyle{T\subseteq N\atop 1\leqslant |T|\leqslant k}}\langle f,v_T\rangle^2=\sum_{\textstyle{T\subseteq N\atop
1\leqslant |T|\leqslant k}}r(f,T)^2.
\end{eqnarray*}

\section*{Acknowledgments}

This research is supported by the internal research project F1R-MTH-PUL-09MRDO of the University of Luxembourg.





\bibliographystyle{elsarticle-num}

\begin{thebibliography}{10}
\bibitem{AloCasHolLor08}
J.~M. Alonso-Meijide, B.~Casas-M{\'e}ndez, M.~J. Holler, and S.~Lorenzo-Freire.
\newblock Computing power indices: multilinear extensions and new
  characterizations.
\newblock {\em European J. Oper. Res.}, 188(2):540--554, 2008.

\bibitem{Ban65}
J.~Banzhaf.
\newblock Weighted voting doesn't work : A mathematical analysis.
\newblock {\em Rutgers Law Review}, 19:317--343, 1965.

\bibitem{ChaGolKeaRou88}
A.~Charnes, B.~Golany, M.~Keane, and J.~Rousseau.
\newblock Extremal principle solutions of games in characteristic function
  form: core, {C}hebychev and {S}hapley value generalizations.
\newblock In {\em Econometrics of planning and efficiency}, volume~11 of {\em
  Adv. Stud. Theoret. Appl. Econometrics}, pages 123--133. Kluwer Acad. Publ.,
  Dordrecht, 1988.

\bibitem{Cur87}
I.~J. Curiel.
\newblock A class of nonnormalized power indices for simple games.
\newblock {\em Math. Social Sci.}, 13(2):141--152, 1987.

\bibitem{DeePac78}
J.~Deegan and E.~W. Packel.
\newblock A new index of power for simple {$n$}-person games.
\newblock {\em Internat. J. Game Theory}, 7(2):113--123, 1978.

\bibitem{DinLaxCheChe08}
G.~Ding, R.~F. Lax, J.~Chen, and P.~P. Chen.
\newblock Formulas for approximating pseudo-{B}oolean random variables.
\newblock {\em Discrete Appl. Math.}, 156(10):1581--1597, 2008.

\bibitem{DinLaxCheChe10}
G.~Ding, R.~F. Lax, J.~Chen, P.~P. Chen, and B.~D. Marx.
\newblock Transforms of pseudo-{B}oolean random variables.
\newblock {\em Discrete Appl. Math.}, 158(1):13--24, 2010.

\bibitem{DubSha79}
P.~Dubey and L.~S. Shapley.
\newblock {Mathematical properties of the {B}anzhaf power index}.
\newblock {\em Math. Oper. Res.}, 4:99--131, 1979.

\bibitem{FujKojMar06}
K.~Fujimoto, I.~Kojadinovic, and J.-L. Marichal.
\newblock Axiomatic characterizations of probabilistic and
  cardinal-probabilistic interaction indices.
\newblock {\em Games Econom. Behav.}, 55(1):72--99, 2006.

\bibitem{Gra97b}
M.~Grabisch.
\newblock Alternative representations of discrete fuzzy measures for decision
  making.
\newblock {\em Internat. J. Uncertain. Fuzziness Knowledge-Based Systems},
  5(5):587--607, 1997.

\bibitem{Gra97a}
M.~Grabisch.
\newblock {$k$}-order additive discrete fuzzy measures and their
  representation.
\newblock {\em Fuzzy Sets and Systems}, 92(2):167--189, 1997.

\bibitem{GraMarRou00}
M.~Grabisch, J.-L. Marichal, and M.~Roubens.
\newblock Equivalent representations of set functions.
\newblock {\em Math. Oper. Res.}, 25(2):157--178, 2000.

\bibitem{GraRou99}
M.~Grabisch and M.~Roubens.
\newblock {An axiomatic approach to the concept of interaction among players in
  cooperative games}.
\newblock {\em Int. J. Game Theory}, 28(4):547--565, 1999.

\bibitem{HamHol92}
P.~Hammer and R.~Holzman.
\newblock {Approximations of pseudo-Boolean functions; applications to game
  theory}.
\newblock {\em Z. Oper. Res.}, 36(1):3--21, 1992.

\bibitem{HamRud68}
P.~Hammer and S.~Rudeanu.
\newblock {\em {Boolean methods in operations research and related areas}}.
\newblock {Berlin-Heidelberg-New York: Springer-Verlag}, 1968.

\bibitem{MarKojFuj07}
J.-L.~Marichal, I.~Kojadinovic, and K.~Fujimoto.
\newblock Axiomatic characterizations of generalized values.
\newblock {\em Discrete Appl. Math.}, 155(1):26--43, 2007.

\bibitem{MurSon93}
T.~Murofushi and S.~Soneda.
\newblock Techniques for reading fuzzy measures (iii): {I}nteraction index (in
  {J}apanese).
\newblock In {\em Proceedings of the 9th Fuzzy Systems Symposium, Sapporo,
  Japan}, pages 693--696, 1993.

\bibitem{Owe72}
G.~Owen.
\newblock Multilinear extensions of games.
\newblock {\em Management Sci.}, 18:P64--P79, 1972.

\bibitem{Owe88}
G.~Owen.
\newblock Multilinear extensions of games.
\newblock In: A.E. Roth, editor.
\newblock {\em {The Shapley Value. Essays in Honor of Lloyd S. Shapley}},
    pages 139--151. {Cambridge University Press}, 1988.

\bibitem{Rou96}
M.~Roubens.
\newblock Interaction between criteria and definition of weights in {MCDA}
  problems.
\newblock In {\em Proceedings of the 44th Meeting of the European Working Group
  "Multiple Criteria Decision Aiding"}, pages 693--696, October 1996.

\bibitem{Sha53}
L.~Shapley.
\newblock {A value for $n$-person games}.
\newblock In {\em {Contributions to the Theory of Games II (Annals of
  Mathematics Studies 28)}}, pages 307--317. {Princeton University Press},
  1953.

\bibitem{ShaShu54}
L.~Shapley and M.~Shubik.
\newblock A method for evaluating the distribution of power in a committee
  system.
\newblock {\em American Political Science Review}, 48:787--792, 1954.

\bibitem{Stra88}
P.~D. Straffin, Jr.
\newblock The {S}hapley-{S}hubik and {B}anzhaf power indices as probabilities.
\newblock In {\em The {S}hapley value}, pages 71--81. Cambridge Univ. Press,
  Cambridge, 1988.

\bibitem{Web88}
R.~J. Weber.
\newblock Probabilistic values for games.
\newblock In {\em The {S}hapley value}, pages 101--119. Cambridge Univ. Press,
  Cambridge, 1988.

\end{thebibliography}



\end{document}